\documentclass[href,12pt]{article}
\usepackage[T2A]{fontenc}
\usepackage[english]{babel}
\usepackage[tbtags]{amsmath}
\usepackage{amsfonts,amssymb}
\usepackage{amssymb}
\usepackage{graphicx}
\usepackage{amscd}
\usepackage{hyperref}
\usepackage{fullpage}
\usepackage{float}

\usepackage{graphics,amsmath,amssymb}
\usepackage{amsthm}
\usepackage{amsfonts}
\usepackage{mathrsfs}
\makeatletter
\def\rdots{\mathinner{\mkern1mu\raise\p@
\vbox{\kern7\p@\hbox{.}}\mkern2mu
\raise4\p@\hbox{.}\mkern2mu\raise7\p@\hbox{.}\mkern1mu}}
\makeatother

\title{DERIVATION OF BELL POLYNOMIALS OF THE SECOND KIND}
\author{Vladimir~Kruchinin}

\begin{document}

\maketitle

\begin{abstract}
New methods for derivation of Bell polynomials of the second kind are presented. The methods are based on an ordinary generating function and its composita. The relation between a composita and a Bell polynomial is demonstrated. Main theorems are written and examples of Bell polynomials for trigonometric functions, polynomials, radicals, and Bernoulli functions are given.
\end{abstract}

\theoremstyle{plain}
\newtheorem{theorem}{Theorem}
\newtheorem{corollary}[theorem]{Corollary}
\newtheorem{lemma}[theorem]{Lemma}
\newtheorem{proposition}[theorem]{Proposition}

\theoremstyle{definition}
\newtheorem{definition}[theorem]{Definition}
\newtheorem{example}[theorem]{Example}
\newtheorem{conjecture}[theorem]{Conjecture}

\theoremstyle{remark}
\newtheorem{remark}[theorem]{Remark}

\newtheorem{Theorem}{Theorem}[section]
\newtheorem{Proposition}[Theorem]{Proposition}
\newtheorem{Corollary}[Theorem]{Corollary}
\theoremstyle{definition}
\newtheorem{Example}[Theorem]{Example}
\newtheorem{Remark}[Theorem]{Remark}
\newtheorem{Problem}[Theorem]{Problem}

\section{Introduction}

	Bell polynomials are an important tool in solving various mathematical problems, among which is finding of higher derivatives of composite functions \cite{WJohnson, LRiordan, Comtet}. However, a general expression for Bell polynomials is rather difficult to derive. One of the main tools in computations of Bell polynomials is exponential generating functions \cite{LRiordan, Comtet}. In this paper, it is proposed to use ordinary generating functions and their compositae \cite{KruSuperposition} to derive expressions for Bell polynomials. Let us introduce the following notation. Let there be given a function $y(x)$ and an ordinary generating function $Y(x,z)=\sum_{n>0} \frac{y^{(n)}(x)}{n!}z^n$. By definition, the Bell polynomial of the second kind is written as
$$
B_{n,k}(y^{(1)},y^{(2)},\ldots y^{(n-k+1)})=\frac{1}{k!}\sum_{\pi_k\in C_n}{n \choose \lambda_1,\lambda_2,\cdots\lambda_k} y^{(\lambda_1)}y^{(\lambda_2)}\cdots y^{(\lambda_k)}
$$
or as
$$
B_{n,k}=\frac{n!}{k!}\sum_{\pi_k\in C_n} \frac{y^{(\lambda_1)}(x)}{\lambda_1!}\frac{y^{(\lambda_2)}(x)}{\lambda_2!}\cdots\frac{y^{(\lambda_k)}(x)}{\lambda_k!}
$$
where $y^{(i)}$ is the $i$-th derivative of the function $y(x)$,
$C_n$ is the set of compositions of $n$, and\\
$\pi_k$ is the composition of $n$ with $k$ parts exactly $\{\lambda_1+\lambda_2+\ldots \lambda_k=n\}$. 

The polynomial $B_{n,k}$ has the form of a triangle in which the left part contains all derivatives of the function $y(x)$ and the right part contains $[y'(x)]^n$. 
$$
\begin{array}{ccccccccccc}
&&&&& y^{(1)}\\
&&&& y^{(2)} && [y^{(1)}]^2\\
&&& y^{(3)} && B_{3,2} && [y^{(1)}]^3\\
&& y^{(4)} && B_{4,2} && B_{4,3} && [y^{(1)}]^4\\
& \rdots && \vdots && \vdots && \vdots && \ddots\\
y^{(n)} && B_{n,2} && \cdots && \cdots && B_{n,n-1} && [y^{(1)}]^n\\
\end{array}
$$ 
The generating function is $Y(x,z)=\sum_{n>0} \frac{y^{(n)}(x)}{n!}z^n=y(x+z)-y(x)$ \cite{WJohnson, LRiordan, Comtet}. Hence, we can introduce the composita of the generating function $Y(x,z)$ as \cite{KruSuperposition, KruCompositae}
$$
Y^{\Delta}(n,k,x)=\sum_{\pi_k\in C_n} \frac{y^{(\lambda_1)}(x)}{\lambda_1!}\frac{y^{(\lambda_2)}(x)}{\lambda_2!}\cdots\frac{y^{(\lambda_k)}(x)}{\lambda_k!},
$$
and the generating function for $Y^{\Delta}(n,k,x)$ will have the expression:
$$
[Y(x,z)]^k=\left(y(x+z)-y(x)\right)^k=\sum_{n\geq k} Y^{\Delta}(n,k,x)z^n
$$
In view of the foregoing, we can write the relation for the Bell polynomial and composita of the ordinary generating function $Y(x,z)$:
\begin{equation}\label{BellFormula}
B_{n,k}=\frac{n!}{k!}Y^{\Delta}(n,k,x).
\end{equation}
Because there is a one-to-one relation between the composita and the Riordan array \cite{KruCompositae}, the exponential Riordan array $(1,y(x))$ and the Bell polynomial $B_{n,k}(y_1,y_2,\ldots, y_{n-k+1})$, where $y(x)=\sum_{n>0} y_n\frac{x^n}{n!}$, are equivalent.

\section{Expressions for Bell polynomials based on the composita of a generating function $Y(\alpha,z)$}

Let us consider the problem of finding the Bell polynomial $B_{n,k}$ as the problem of finding coefficients of an ordinary generating function $Y(\alpha,z)^k$. This is possible if we represent the generating function $Y(x,z)$ as $F(g(x),h(z))$; then we can use the apparatus of compositae introduced in \cite{KruSuperposition, KruCompositae}. Let us consider the following examples:

\Example
Let there be given a function $y(x)$ with two derivatives $y'(x)$ and $y''(x)$. Let us find an expression for the composita of this function. By definition,
$$
Y^{\Delta}(n,k,x)=\sum_{\pi_k\in C_n}\frac{y^{(\lambda_1)}}{\lambda_1!}\frac{y^{(\lambda_2)}}{\lambda_2!}\cdots \frac{y^{(\lambda_k)}}{\lambda_k!}.
$$
Then, the generating function has the expression $y(x+z)-y(x)=y'(x)z+\frac{y''(x)}{2}z^2$. 
Hence, according to the formula of the composita for the polynomial $ax+bx^2$ \cite{KruSuperposition}, we obtain
\begin{equation} \label{fk_f1f2}
Y^{\Delta}(n,k,x)={k \choose n-k}[f'(x)]^{2k-n}\left(\frac{f''(x)}{2}\right)^{n-k}.
\end{equation}
Thus, the Bell polynomial for the function with derivatives $y'(x)$ and $y''(x)$ is equal to
\begin{equation} 
B_{n,k}=\frac{n!}{k!}{k \choose n-k}[f'(x)]^{2k-n}\left(\frac{f''(x)}{2}\right)^{n-k}.
\end{equation}

\Example\label{Example1} Let there be given a function $y(x)=x^{m}$, where $m>0$. The generating function is $Y(x,z)=(x+z)^m-x^m$. Let us find a composita of $Y(x,z)$; for this purpose, we are to find the coefficients:
$$
x^{km}\left[\left(1+\frac{z}{x}\right)^m-1\right]^k=x^{km}\sum_{j=0}^k{k \choose j}\left(1+\frac{z}{x}\right)^{jm}(-1)^{k-j};
$$
From whence, knowing that the coefficients for $\left(1+\frac{z}{x}\right)^{jm}$ are equal to ${jm \choose n}\frac{1}{x^n}$, we obtain the desired composita
$$
Y^{\Delta}(n,k,x)=x^{km}\sum_{j=0}^k{k \choose j}{jm \choose n}x^{-n}(-1)^{k-j}.
$$
Then the Bell polynomial is
$$
B_{n,k}=\frac{n!}{k!}x^{km-n}\sum_{j=0}^k{k \choose j}{jm \choose n}(-1)^{k-j}.
$$
\Example Let there be given a function $y(x)=x^{-m}$, where $m>0$. The generating function is $Y(x,z)=\frac{1}{(x+z)^m}-\frac{1}{x^m}$. Let us find a composita of $Y(x,z)$; for this purpose, we are to find the coefficients
$$
Y(x,z)^k=\frac{1}{x^{mk}}\left[\frac{1}{\left(1+\frac{z}{x}\right)^m}-1\right]^k;
$$
from whence it follows that the composita is equal to
$$
\left(\sum_{j=1}^{k}{{{k}\choose{j}}\,\left(-1\right)^{n+k-j}\,{{n+
 j\,m-1}\choose{j\,m-1}}}\right)\,x^{-n-k\,m}.
$$
Then the Bell polynomial is equal to
$$
B_{n,k}=\frac{n!}{k!}\left(\sum_{j=1}^{k}{{{k}\choose{j}}\,\left(-1\right)^{n+k-j}\,{{n+
 j\,m-1}\choose{j\,m-1}}}\right)\,x^{-n-k\,m}.
$$

%\subsection{Characteristic polynomial of $\sin x$}
\Example Let us consider the example of use of the composita for the generating function $f(z)=az+bz^2+cz^3$: 
$$
F^{\Delta}(n,k)=\sum\limits_{j=0}^k{k\choose j}{j \choose n-k-j}a^{k-j}b^{2j+k-n}c^{n-k-j}.
$$ 
Substitution of $a=\frac{f'(x)}{1!}$, $b=\frac{f''(x)}{2!}$, $c=\frac{f'''(x)}{3!}$ gives the Bell polynomial:
$$
B_{n,k}=\frac{n!}{k!}\sum\limits_{j=0}^k{k\choose j}{j \choose n-k-j}(f'(x))^{k-j}\left(\frac{f''(x)}{2}\right)^{2j+k-n}\left(\frac{f'''(x)}{6}\right)^{n-k-j}.
$$ 
Let us consider the example $f(x)=x^3+2x$, $f'(x)=3x^2+2$, $f''(x)=6x$, $f'''(x)=6$; then, $a=3x^2+2$, $b=3x$, $c=1$. Then the Bell polynomial is
$$\frac{n!}{k!}\sum_{j=0}^{k}{{{j}\choose{n-k-j}}\,{{k}\choose{j}}\,3^{-n+k+2\,j}
 \,x^{-n+k+2\,j}\,\left(3\,x^2+2\right)^{k-j}}$$
Presented below are the first terms of this polynomial
$$3x^2+2$$
$$6x,~(3x^2+2)^2$$
$$6,~18x(3x^2+2),~(3x^2+2)^3$$
$$0,~180x^2+48,~36x(3x^2+2)^2,~(3x^2+2)^4$$

The same reasoning allows us to obtain Bell polynomials for functions whose generating functions $y(x+z)-y(x)$ are expressed in polynomials. Expressions for the compositae of polynomials and methods of their derivation are described in \cite{KruSuperposition}.

\Example
Let us find a Bell polynomial for the function $\sin x$. For this purpose, we find the composita of the function $\sin(x+z)-\sin x$. Then
$$
S(x,z)=\cos x\sin z+\sin x(\cos z-1),
$$
where $\sin z$ and $\cos z$ are generating functions, and $\sin x$ and $\cos x$ are coefficients. Hence the composita of the function $\cos x\sin z$ \cite{KruSuperposition} is
$$
F^{\Delta}(n,k,x)=(\cos x)^k\frac{(1+(-1)^{n-k})}{2^{k}n!}\sum\limits_{m=0}^{\frac{k}{2}} {k \choose m} (2m-k)^n(-1)^{\frac{n+k}{2}-m},
$$
Now, let us write the coefficients $T_{n,k}$ for $\cos^k(z)=\sum_{n\geq 0} T_{n,k}z^n$.
$$
T_{n,k}=\left\{
\begin{array}{ll}
1,& n=0\\
0,& n- \hbox{odd},\\
\frac{1}{2^{k-1}}\sum_{i=0}^{\frac{k-1}{2}}{k\choose i}\frac{(k-2i)^{n}}{(n)!}(-1)^\frac{n}{2},& n- \hbox{even}.\\
\end{array}
\right.
$$
Then we obtain the composita of the generating function $\sin x(\cos(z)-1)$ 
$$
R^{\Delta}(n,k,x)=(\sin x)^k\frac{(-1)^n+1}{n!}\sum_{j=1}^{k}\frac{(-1)^{{{n}\over{2}}+k-j}}{2^j}{{k}\choose{j}}\sum_{i=0}^{\left \lfloor {{j-1}\over{2}} \right \rfloor}{\left(j-2i\right)^{n}{j\choose i}}$$

Next, from the theorem of the composita of the sum of generating functions \cite{KruSuperposition}, we obtain the desired composita
{
$$
S^{\Delta}(n,k,x)=F^{\Delta}(n,k,x)+R^{\Delta}(n,k,x)+\sum_{j=1}^{k-1}{k \choose j}
\sum_{i=j}^{n-k+j}F^{\Delta}(i,j,x)R^{\Delta}(n-i,k-j,x).
$$
}
Presented below are the first terms of the Bell polynomial $B_{n,k}=\frac{n!}{k!}S^{\Delta}(n,k,x)$ for the function $\sin x$:
$$\cos x$$
$$-\sin x,~\cos ^2x$$
$$-\cos x,~-3\,\cos x\,\sin x,~\cos ^3x$$
$$\sin x,~3\,\sin ^2x-4\,\cos ^2x,~-6\,\cos ^2x\,\sin x,~\cos ^4x$$
$$\cos x,~15\,\cos x\,\sin x,~15\,\cos x\,\sin ^2x-10\,\cos ^3x,~-10\,\cos ^3x\,\sin x,~\cos ^5x$$

Now the derivative $f_1^{(4)}(x)$ for the function $f_1(x)=e^{\sin x}$ is expressed as
$$
f_1^{(4)}(x)=e^{\sin x}\left(\sin x+3\sin^2x-4\cos^2x-6\cos^2x\sin x+\cos^4x\right).
$$
The derivative $f_2^{(5)}(x)$ for $f_2(x)=\sin^3 x$ is expressed as
$$
f_2^{(5)}(x)=3\sin^2 x(\cos x)+6\sin x(15\cos x\sin x)+6(15\cos x\sin^2x-10\cos^3x)=
$$
$$
=183\sin^2 x\cos x-60\cos^3x.
$$

In the same way, we can find a Bell polynomial for the function $\cos x$; for this purpose, we are to find the composita of the generating function:
$$
C(x,z)=\cos x(\cos z-1)-\sin x \sin z.
$$

\Example Let us consider the function $y(x)=\sqrt[3]{x}$. The generating function is $Y(x,z)=\sqrt[3]{x+z}-\sqrt[3]{x}$. Hence
$$
Y(x,z)^m=(-1)^m(\sqrt[3]{x})^m\left[1-\sqrt[3]{\left(1+\frac{z}{x}\right)}\right]^m.
$$
Given the composita of the generating function $1-\sqrt[3]{1-z}$ \cite{KruCompositae}, we obtain the desired composita 
$$
Y^{\Delta}(n,m,x)=\left\{\begin{array}{ll}	
(\sqrt[3]{x})^m(\frac{1}{3})^n, & n=m,\\
(\sqrt[3]{x})^m{\frac{m}{n}\sum\limits_{k=1}^{n-m}{{{k}\choose{n-m-k}}3^{-2n+m+k}(-1)^{k}\,{{n+k-1}\choose{n-1}}}}x^{-n}, & n>m.
\end{array}
\right. 
$$

\section{Method based on operations on compositae $Y^{\Delta}(n,k,x)$}

Let us consider peculiarities of the generating function $Y(x,z)=\sum_{n>0} \frac{y^{(n)}(x)}{n!}z^n=y(x+z)-y(x)$. For this purpose, we prove the following theorem.
\begin{Theorem}\label{FaDeeBruno}
Let there be given a composition $f(x)=g(y(x))$ and functions $g(x)$, $y(x)$ with an infinite number of derivatives in the general case. Then the generating functions $F(x,z)=\sum_{n\geqslant 0}\frac{f^{(n)}(x)}{n!}z^n$, $Y(x,z)=\sum_{n\geqslant 1}\frac{y^{(n)}(x)}{n!}z^n$ и $G(x,z)=\sum_{n\geqslant 0}\frac{g^{(n)}(x)}{n!}z^n$
form the composition
$$
F(x,z)=G(y,Y(x,z)).
$$
\end{Theorem}
\begin{proof} Let us write the known Faa di Bruno formula \cite{WJohnson, LRiordan, Comtet}:
$$
f^{(n)}(x)=\sum_{k=1}^ng^{(k)}(y)\frac{n!}{k!}\sum_{\pi_k\in C_n} \frac{y^{(\lambda_1)}(x)}{\lambda_1!}\frac{y^{(\lambda_2)}(x)}{\lambda_2!}\cdots\frac{y^{(\lambda_k)}(x)}{\lambda_k!}.
$$
Hence
\begin{equation}\label{CompositionOGF}
\frac{f^{(n)}(x)}{n!}=\sum_{k=1}^n\frac{g^{(k)}(y)}{k!}\sum_{\pi_k\in C_n} \frac{y^{(\lambda_1)}(x)}{\lambda_1!}\frac{y^{(\lambda_2)}(x)}{\lambda_2!}\cdots\frac{y^{(\lambda_k)}(x)}{\lambda_k!}
\end{equation} 
Thus, we obtain the formula for the composition of ordinary generating functions \cite{KruSuperposition}. It is evident that the nonzero term of $F(x,z)$ is equal to $g(y(x))$. 
\end{proof} The peculiarity here is that in the operation of the composition of generating functions, the argument $x$ in $G(x,z)$ is replaced by $y(x)$.
%\begin{}
Let us turn to the problem of finding compositae of the generating functions $\left(y(x+z)-y(x)\right)$ using the operations of summation, product, and composition.

\begin{Theorem} \label{Theorem_sum} Let there be generating functions $F(x,z)=f(x+z)-f(x)=\sum\limits_{n>0} \frac{f^{(n)}(x)}{n!}z^n$, $G(x,z)=g(x+z)-g(x)=\sum\limits_{n>0} g^{(n)}(x)z^n$  and their compositae $F^{\Delta}(n,k,x)$ , $G^{\Delta}(n,k,x)$. Then the generating function $A(x,z)=F(x,z)+G(x,z)$ has the composita
$$
A^{\Delta}(n,k,x)=F^{\Delta}(n,k,x)+\sum\limits_{j=1}^{k-1}{k\choose j}\sum\limits_{i=j}^{n-k+j}F^{\Delta}(i,j,x)G^{\Delta}(n-i,k-j,x)+G^{\Delta}(n,k,x).
$$
\end{Theorem}
\begin{proof} without proof \end{proof}
\Example
Let there be $f(x)=x^2$, $F(x,z)=2xz+z^2$, a composita $F^{\Delta}(n,k,x)={k \choose n-k}(2x)^{2k-n}$ and
      $g(x)=\ln(x)$, $G(x,z)=\ln(x+z)-\ln(x)=\ln(1+\frac{z}{x})$, and a composita $G^{\Delta}(n,k,x)=\frac{k!}{n!}\left[{n\atop k}\right]x^{-n}$. Then for the function $a(x)=x^2+\ln(x)$, the Bell polynomial is
$$
B_{n,k}=\frac{n!}{k!}\sum_{j=0}^{k} {k \choose j}\sum_{i=j}^{n-k+j}\frac{j!}{i!}\left[{i\atop j}\right]{k-j \choose n-i-k+j}2^{2(k-j)-n+i}x^{2(k-j)-n}.
$$
Now let us turn to finding of the composita $Y^{\Delta}(n,k,x)$ of the function $y(x)=f(x)g(x)$ expressed as the product of the functions $f(x)$ and $g(x)$. Let us prove the following theorem.
\begin{Theorem} \label{Theorem_prod} Let there be a function $a(x)=f(x)g(x)$; then the composita of the function $Y(x,z)=f(x+z)g(x+z)-f(x)g(x)$ is equal to
$$
Y^{\Delta}(n,k,x)=\sum_{j=0}^k {k \choose j}\left( \sum_{i=0}^n F(i,j,x)G(n-i,j,x)\right)[f(x)g(x)]^{k-j}(-1)^{k-j}. 
$$
where $F(n,k)$ are coefficients of the generating function $[f(x+z)]^k$, and $G(n,k)$ -- $[g(x+z)]^k$.
\end{Theorem}
\begin{proof} Here we have the second peculiarity: it is necessary to take into account the rule of finding a derivative of the product. According to the Leibniz rule, we can write
$$
\frac{y^{(n)}}{n!}=\sum_{i=0}^n \frac{f^{(i)}}{i!}\frac{g^{(n-i)}}{(n-i)!}.
$$
Hence 
$$
y(x+z)=f(x+z)g(x+z)
$$
Now let us find coefficients for the expression $[f(x+z)g(x+z)-f(x)g(x)]^k$. By removing the brackets and substituting the expression for the coefficients of the generating functions $f(x+z)$ and $g(x+z)$, we obtain the desired formula.
\end{proof}
Given the composita of the generating function $f(x+z)-f(x)$ -- $F^{\Delta}(n,k,x)$, the coefficients of the generating function $[f(x+z)]$ are calculated by the formula:
$$
F(n,k,x)=\sum_{j=0}^k {k \choose j} F^{\Delta}(n,j,x)f(x)^{k-j}.
$$

\Example Let there be a function $x^{ax}$ (see the example in \cite{Comtet} ). Let us find an expression for the $n$-th derivative of this function. Let us write it in the form $\exp(ax\ln(x))$. For this purpose, we find a composita of the function $(x+z)\ln(x+z)-x\ln(x)$ and expressions for coefficients of the generating functions $(x+z)^k$ and $\ln(x+z)^k$. For the first function, $F(n,k)={k \choose n}x^{k-n}$; for the second function, $G(n,k)=\sum_{j=0}^k{k \choose j} \frac{j!}{n!}\left[{n\atop j}\right]x^{-n}\ln(x)^{k-j}$. Then the composita of the function $(x+z)\ln(x+z)-x\ln(x)$ is

$$
A^{\Delta}(n,k,x)=x^{k-n}\,\sum_{j=0}^{k}{\left(-1\right)^{k-j}\,{{k}\choose{j}}\,
 \left(\sum_{i=0}^{n}{{{{{j}\choose{i}}\,\frac{1}{(n-i)!}\sum_{m=0}^{j}{m!\,{{j
 }\choose{m}}\,{\left[n-i \atop m\right]}}}}}\,\left(\ln x\right)^{k-m}\right)}.
$$
From this it follows that the composita of the function $ax\ln x$ is equal to $a^kA^{\Delta}(n,k,x)$. Presented below are the first terms of this composita.
$$a(\ln x+1)$$
$${{a}\over{2\,x}},~~a^2(\ln x+1)^2$$
$$-{{a}\over{6 x^2}},~~{{a^2\,\ln x+a^2}\over{x}},~~a^3(\ln x+1)^3$$
$${{a}\over{12\,x^3}},~~ -{{4\,a^2\,\ln x+a^2}\over{12\,x^2}},~~ {{3\,a^3\,\ln ^2x+6\,a^3\,\ln x+3\,a^3}\over{2\,x}},~~a^4(\ln x+1)^4$$
 
Hence the expression for the $n$-th derivative of the generating function $x^{ax}$ has the form:
$$
[x^{ax}]^{(n)}=x^{ax}\sum_{k=1}^n \frac{n!}{k!}a^k x^{k-n} \sum_{j=0}^{k}{\left(-1\right)^{k-j}\,{{k}\choose{j}}\left(\sum_{i=0}^{n}{{{{{j}\choose{i}}\sum_{m=0}^{j}{\frac{m!}{(n-i)!}{{j
 }\choose{m}}{\left[n-i \atop m\right]}}}}}\left(\ln x\right)^{k-m}\right)}.
$$

Now let us consider the operation of product of compositae. %with regard to the theorem \ref{FaDeeBruno}.
For this purpose, we prove the following theorem.
\begin{Theorem} \label{Theorem_composition} Let there be functions $f(x)$, $g(x)$ and compositae of the generating functions $F^{\Delta}(n,k,x)$ for $f(x+z)-f(x)$ and $G^{\Delta}(n,k,x)$ for $g(x+z)-g(x)$. Then for the composition of the functions $y(x)=g(f(x))$, the composita of the generating function $Y(x,z)=g(f(x+z))-g(f(x))$ is
$$
Y^{\Delta}(n,m,x)=\sum_{k=m}^n F^{\Delta}(n,k,x)G^{\Delta}(k,m,f(x)). 
$$
\end{Theorem}
\begin{proof} 
$$
[f(x+z)-f(x)]^m=\sum_{n\geq m} F^{\Delta}(n,m,x)z^n
$$
From formula (\ref{CompositionOGF}) we have
$$
Y^{\Delta}(n,m,x)=\sum_{k=m}^n F^{\Delta}(n,k,x)G^{\Delta}(k,m,f(x)). 
$$
Given the expression for the coefficients $Y(n,k,x)$ of the generating function $y(x+z)^k$, the expression for the coefficients of the composition of the generating functions $a(x+z)=y(f(x+z)$ has the form:
$$
A(n,m,x)=\left\{
\begin{array}{ll}
y(f(x))^m, & n=0\\
\sum_{k=1}^n F^{\Delta}(n,k,x)Y(k,m,x), & n>0.
\end{array}
\right.
$$
\end{proof}
Note that this theorem holds true for Bell polynomials as well \cite{Comtet}, because
$$
B_{n,m}(x)=\sum_{k=m}^n \frac{n!}{k!}F^{\Delta}(n,k,x)\frac{k!}{m!}G^{\Delta}(k,m,f(x))=\frac{n!}{m!}\sum_{k=m}^n F^{\Delta}(n,k,x)G^{\Delta}(k,m,f(x)). 
$$

\Example\label{example_1vln}
Let us find a composita of the function $f(x)=\frac{1}{x}$. The generating function for the composita is $F(x,z)=\frac{1}{x+z}-\frac{1}{x}=\frac{1}{x}\frac{-\frac{z}{x}}{1+\frac{z}{x}}$. Hence
$$
F^{\Delta}(n,k,x)={n-1\choose k-1}(-1)^{n}x^{-n-k}.
$$
Now let us write the composition $a(x)=\frac{1}{\ln(x)}$. The composita for the generating function $\ln(x+z)-\ln(x)$ is  
$\frac{k!}{n!}{\left[n \atop k\right]x^{-n}}$. From this it follows that the desired composita is equal to
$$
A^{\Delta}(n,m)=\sum_{k=m}^n \frac{k!}{n!}{\left[n \atop k\right]x^{-n}}{k-1\choose m-1}(-1)^{k}(\ln(x))^{-n-k},
$$
and the Bell polynomial is
$$
B_{n,k}=m!\sum_{k=m}^n k!{\left[n \atop k\right]x^{-n}}{k-1\choose m-1}(-1)^{k}(\ln(x))^{-n-k}.
$$
\Example Let us find a Bell polynomial for the function $a(x)=\frac{1}{1-x-x^2}$, the function $a(x)=g(f(x))$, where $g(x)=\frac{1}{1-x}$, $f(x)=x+x^2$. The composita of the function $f(x)$ is equal to $F^{\Delta}(n,k,x)={k \choose{n-k}}(2x+1)^{2k-n}$ (see example No. \ref{Example1}). The composita of the function $g(x)=\frac{1}{1-x}$ is equal to $F^{\Delta}(n,k,x)={n-1 \choose{k-1}}(1-x)^{-k-n}$. Using theorem \ref{Theorem_composition}, we obtain the desired Bell polynomial:
$$
B_{n,m}=\frac{n!}{m!}\sum_{k=m}^{n} {{k-1}\choose{m-1}}{{k}\choose{n-k}}(2x+1)^{2k-n}(1-x-x^2)^{-m-k}.
$$
$${{2\,x+1}\over{\left(-x^2-x+1\right)^2}}$$
$${{2}\over{\left(-x^2-x+1\right)^2}}+{{2\left(2\,x+1\right)^
 2}\over{\left(-x^2-x+1\right)^3}},~
{{\left(2\,x+1\right)^2}\over{\left(-x^2-x+1\right)^4}}$$
$${{12\,\left(2\,x+1\right)}\over{\left(-x^2-x+1\right)^3}}+
 {{6\left(2\,x+1\right)^3}\over{\left(-x^2-x+1\right)^4}},~
{{6\,\left(2\,x+1\right)}\over{\left(-x^2-x+1\right)^4}}+
 {{6\left(2\,x+1\right)^3}\over{\left(-x^2-x+1\right)^5}},~
{{\left(2\,x+1\right)^3}\over{\left(-x^2-x+1\right)^6}}$$

\Example Let us find a Bell polynomial for the function $\tan(x)$. For this purpose, we represent the generating function as
$$
A(x,z)=\tan(x+z)-\tan(x)=\frac{\tan(x)+\tan(z)}{1-\tan(x)\tan(z)}-\tan(x)=\frac{\tan(z)\sec(x)^2}{1-\tan(x)\tan(z)}.
$$
Hence $A(x,z)=f(x,\tan(z))$, where $f(x,z)=\frac{\sec(x)^2z}{1-\tan(x)z}$. Then the composita of $f(x,z)$ is equal to
$$
F^{\Delta}(n,k,x)={{n-1}\choose{k-1}}\tan(x)^{n-k}\sec(x)^{2k}.
$$
The composita of the generating function $\tan(z)$ is
$$
G^{\Delta}(n,k)=\frac{1+(-1)^{n-k}}{n!}\sum\limits_{j=k}^n 2^{n-j-1}\left\{{n \atop j}\right\}j!(-1)^{\frac{n+k}{2}+j}{j-1 \choose k-1}.
$$
Using the theorem of product of compositae \cite{KruCompositae}, we obtain the composita of the desired function:
$$
G^{\Delta}(n,m)=\sum_{k=m}^n G^{\Delta}(n,k)F^{\Delta}(k,m)=
$$
$$
=\sum_{k=m}^n \frac{1+(-1)^{n-k}}{n!}\sum\limits_{j=k}^n 2^{n-j-1}\left\{{n \atop j}\right\}j!(-1)^{\frac{n+k}{2}+j}{j-1 \choose k-1} {{{k-1}\choose{m-1}}\tan(x)^{k-m}\sec(x)^{2m}.}
$$
Hence the Bell polynomial is equal to
$$
B_{n,m}=\frac{\sec(x)^{2m}}{m!}\sum_{k=m}^n \frac{1+(-1)^{n-k}}{2}\sum\limits_{j=k}^n 2^{n-j}\left\{{n \atop j}\right\}j!(-1)^{\frac{n+k}{2}+j}{j-1 \choose k-1} {{{k-1}\choose{m-1}}\tan(x)^{k-m}.}
$$
$$\sec(x)^2$$
$$2\sec(x)^2\tan(x),~\sec(x)^4$$
$$6\sec(x)^2\tan(x)^2+2\sec(x)^2,~6\sec(x)^4\tan(x),\sec(x)^6]$$
$$24\sec(x)^2\tan(x)^3+16\sec(x)^2\tan(x),36\sec(x)^4\tan(x)^2+8\sec(x)^4,12\sec(x)^6\tan(x),\sec(x)^8$$

Given the composita of the function $\tan(x)$, we can obtain the composita of $\cot(x)$ by representing $\cot(x)=\frac{1}{\tan(x)}$ (see example \ref{example_1vln}).

\Example Let us derive a Bell polynomial for the function $\arctan(x)$. For this purpose, we write the generating function
$$
A(x,z)=\arctan(x+z)-\arctan(x)=\arctan\left(\frac{z}{1+x^2+xz}\right).
$$
Let us find a composita of the function $\frac{z}{1+x^2+xz}$. We represent it as
$$
f(x,z)=\frac{1}{(1+x^2)}\frac{z}{1+\frac{xz}{1+x^2}}.
$$
Hence, the composita of the function $f(x,z)$ is equal to
$$
F^{\Delta}(n,k)={n-1 \choose k-1}(-1)^{n-k}\frac{x^{n-k}}{(1+x^2)^n}.
$$
Given the composita of the generating function $\arctan(z)$ \cite{KruCompositae} 
$$
\frac{\left((-1)^{\frac{3n+k}{2}}+(-1)^\frac{n-k}{2}\right)k!}{2^{k+1}}\sum\limits_{j=k}^n \frac{2^j}{j!}{n-1 \choose j-1}\left[{j \atop k}\right],
$$
we obtain the composita of the desired generating function $A(x,z)$:
$$
A^{\Delta}(n,m)=\sum_{k=m}^n {n-1 \choose k-1}\frac{(-x)^{n-k}}{(1+x^2)^n}\frac{\left((-1)^{\frac{3k+m}{2}}+(-1)^\frac{k-m}{2}\right)m!}{2^{m+1}}\sum\limits_{j=m}^k \frac{2^j}{j!}{k-1 \choose j-1}\left[{j \atop m}\right].
$$
Hence the desired Bell polynomial is equal to
$$
B_{n,m}=n!\sum_{k=m}^n {n-1 \choose k-1}\frac{(-x)^{n-k}}{(1+x^2)^n}\frac{\left((-1)^{\frac{3k+m}{2}}+(-1)^\frac{k-m}{2}\right)}{2^{m+1}}\sum\limits_{j=m}^k \frac{2^j}{j!}{k-1 \choose j-1}\left[{j \atop m}\right].
$$
Presented below are the first terms of the Bell polynomial for the function $\arctan(x)$
$${{1}\over{x^2+1}}$$
$$-{{2\,x}\over{\left(x^2+1\right)^2}},~{{1}\over{\left(x^2+1\right)^2}}$$
$$6\,\left({{x^2}\over{\left(x^2+1\right)^3}}-{{1}\over{3\,\left(x^2+
 1\right)^3}}\right),~-{{6\,x}\over{\left(x^2+1\right)^3}},~{{1}\over{\left(x^2+1\right)^3}}$$
$$24\,\left({{x}\over{\left(x^2+1\right)^4}}-{{x^3}\over{\left(x^2+1
 \right)^4}}\right),~12\left({{3\,x^2}\over{\left(x^2+1\right)^4}}-{{2}\over{3\,\left(
 x^2+1\right)^4}}\right),~-{{12\,x}\over{\left(x^2+1\right)^4}},~{{1}\over{\left(x^2+1\right)^4}}$$

\Example Let us find a Bell polynomial for the function $a(x)=\frac{x}{\sqrt{1-x^2}}$. For this purpose, we represent this function in the form $g(h(g(f(x)))=\frac{1}{\sqrt{\frac{1}{x^2}-1}}$. Let us write the compositae for the functions $f(x)=x^2$ and $g(x)=\frac{1}{x}$ 
$$
F^{\Delta}(n,k,x)={k \choose n-k}(2x)^{2k-n}
$$
$$
G^{\Delta}(n,k,x)={n-1\choose k-1}(-1)^{n}x^{-n-k}.
$$
Hence the composita of $g(f(x)=\frac{1}{x^2}$ is equal to
$$
x^{-n-2k}\sum_{k=m}^n {k \choose n-k}2^{2k-n}{k-1\choose m-1}(-1)^{k}.
$$
Now let us find a composita of the function $\frac{1}{\sqrt{x}}$. For this purpose, we also use the composition of the functions $g(h(x))$.
Let us derive a composita for the function $\sqrt{x}$. For this purpose, we write the generating function $\sqrt{x+z}$ 
$$
H(x,z)=\sqrt{x+z}-\sqrt{x}=-\sqrt{x}2\frac{(1-\sqrt{1-4\frac{z}{4x}})}{2}.
$$
Note that the generating function in the brackets is the generating function for Catalan numbers \cite{KruCompositae}. Given the composita of the function, we obtain the composita for $H(x,z)$
$$
H^{\Delta}(n,k,x)=\frac{k}{n}{2n-k-1 \choose n-1}(-1)^{n-k}(\sqrt{x})^k2^k4^{-n}.
$$
Hence the composita of the function $\frac{1}{\sqrt{x}}$ is equal to
$$
(-1)^n(\sqrt{x})^m4^{-n}\sum_{k=m}^n \frac{k}{n}{2n-k-1 \choose n-1}2^k{k-1\choose m-1}.
$$
This result was obtained by L. Comtet \cite{Comtet}.
Now from theorem \ref{Theorem_composition}, we obtain the composita of the function $a(x)=\frac{x}{\sqrt{1-x^2}}$
$$
x^{m-n}\sum_{k=m}^{n}{{{\frac{(-1)^{k}}{k4^k}\sum_{j=m}^{k
 }{j2^{j}{{j-1}\choose{m-1}}{{2\,k-j-1}\choose{k-1}}}
 \sum_{i=k}^{n}{(-1)^{i}{{i-1}\choose{k-1}}{{i
 }\choose{n-i}}2^{2i-n}}(1-x^2)^{-{{m}\over{2
 }}-k}}}}$$

\Example Let us find a Bell polynomial for the generating function of Bernoulli numbers $a(x)=\frac{x}{e^x-1}$. For this purpose, we write the expressions for the coefficients of the generating functions $F(x,z)=(x+z)^k$ and $G(x,z)=\left(\frac{1}{x+z-1}\right)^k$. Hence
$$
F(n,k,x)={k\choose n}x^{k-n}.
$$
$$ 
G(n,k,x)={n+k-1\choose k-1}(x-1)^{-n-k}(-1)^n.
$$
Using formula (\ref{CompositionOGF}) for the composition of the generating functions $g(x+z)^m$ and $e^{x+z}$, we obtain expressions for the coefficients of $h(x+z)=[\frac{1}{e^{(x+z)}-1}]^m$
$$
H(n,m,x)=\left\{
\begin{array}{ll}
\frac{1}{(e^x-1)^m}, & n=0\\
\frac{1}{n!}{{\sum_{k=0}^{m}{\left(-1\right)^{k}\,k!\,{{m+k-1}\choose{m-1}}\,
 \left\{n \atop k\right\}\,\left(e^{x}-1\right)^{-m-k}\,e^{
 k\,x}}}}, & n>0.
\end{array}
 \right.
$$

From theorem \ref{Theorem_prod} we obtain the composita of the product $x\frac{1}{e^x-1}$
$$\sum_{j=0}^{m}{\left(-1\right)^{m-j}\,{{m}\choose{j}}\,\left({{x
 }\over{e^{x}-1}}\right)^{m-j}\,\sum_{i=0}^{n}{H(i,j,x)\,{{j}\choose{n-i}}\,x^{-n+j+i}}}.
$$
Then the Bell polynomial for the generating function of Bernoulli numbers has the form:
$$
B_{n,m}=\frac{n!}{m!}\sum_{j=0}^{m}{\left(-1\right)^{m-j}\,{{m}\choose{j}}\,\left({{x
 }\over{e^{x}-1}}\right)^{m-j}\,\sum_{i=0}^{n}{H(i,j,x)\,{{j}\choose{n-i}}\,x^{-n+j+i}}}.
$$

\section{Bell polynomials of inverse functions}

\begin{Theorem}\label{theorem_inversion} Let there be given a function $f(x)$ and its composita $F^{\Delta}(n,m,x)$. For the composita $Y^{\Delta}(n,m,x)$ of the inverse function $f^{-1}(x)=y(x)$, the following recurrent expressions hold true:
\begin{equation}\label{inversion1}
Y_1^{\Delta}(n,m,x)=\left\{
\begin{array}{ll}
\frac{1}{F^{\Delta}(m,m,y(x))}& n=m,\\
-\frac{1 }{F^{\Delta}(m,m,g(x))}\sum\limits_{k=m+1}^n Y^{\Delta}(n,k,x)F^{\Delta}(k,m,y(x))& n>0.
\end{array}
\right.
\end{equation}

\begin{equation}\label{inversion2}
Y_2^{\Delta}(n,m,f(x))=\left\{
\begin{array}{ll}
\frac{1}{F^{\Delta}(n,n,x)}& n=m,\\
-\frac{1}{F^{\Delta}(n,n,x)}\sum\limits_{k=m}^{n-1} F^{\Delta}(n,k,x)Y^{\Delta}(k,m,f(x))& n>0.
\end{array}
\right.
\end{equation}
\end{Theorem}
\begin{proof}
For self-inverse functions, the condition
$$
f(f^{-1}(x))=f^{-1}(f(x))=x
$$
is fulfilled. Hence from theorem \ref{Theorem_composition}, we can write
$$
\sum_{k=m}^n Y^{\Delta}(n,k,x)F^{\Delta}(k,m,y(x))=\sum_{k=m}^{n} F^{\Delta}(n,k,x)Y^{\Delta}(k,m,f(x))=\delta(n,m).
$$
Simple transformations give us formulae (\ref{inversion1},~\ref{inversion2}).
\end{proof}

Now from formula (\ref{BellFormula}) we can write the Bell polynomial of the inverse function
$$
B_{n,m}=\frac{n!}{m!}Y_1^{\Delta}(n,m,x)=\frac{n!}{m!}Y_2^{\Delta}(n,m,f(x)).
$$ 
\Example Let us consider a simple example. Let there be a function $f(x)=x^2$, its composita $F^{\Delta}(n,k,x)={k \choose n-k}(2x)^{2k-n}$, and inverse function $g(x)=\sqrt{x}$.  Let us find an expression for the Bell polynomial of the function $\sqrt{x}$, given the composita of the function $f(x)=x^2$. In view of expression (\ref{inversion1}), we obtain
$$
Z_1^{\Delta}(n,m,x)=\left\{
\begin{array}{ll}
\frac{1}{(2\sqrt{x})^m}, & m=n\\
-\frac{1}{2^m\sqrt{x}^m}\sum_{k=m+1}^{n}Z_1^{\Delta}(n,k,x){m \choose k-m}(2\sqrt{x})^{2m-k}, & n>m.
\end{array}
\right.
$$

In view of expression (\ref{inversion2}), we derive
$$
Z_2^{\Delta}(n,m,x)=\left\{
\begin{array}{ll}
\frac{1}{(2x)^n}, & m=n\\
-\frac{1}{2^nx^n}\sum_{k=m}^{n-1}{k \choose n-k}(2x)^{2k-n}Z_2^{\Delta}(k,m,x), & n>m.
\end{array}
\right.
$$
Hence the Bell polynomial for the function $\sqrt{x}$ is equal to
$$
B_{n,m}=\frac{n!}{m!}Z_1^{\Delta}(n,m,x)=\frac{n!}{m!}Z_2^{\Delta}(n,m,\sqrt{x}).
$$ 

\Example Let there be a function $f(x)=x\exp(x)$ and Lambert function $W(x)$. Let us find an expression for the n-derivative of the function $W(f(x))$. From theorem \ref{Theorem_prod} the composita of the function $f(x)$ is equal to
$$
F^{\Delta}(n,k,x)=e^{k\,x}\,\sum_{i=0}^{n}{{{k^{n-i}\,{{k}\choose{i}}\,x^{k-i}}\over{
 \left(n-i\right)!}}}
$$
and the Bell polynomial is equal to

$$
B_{n,k}=\frac{n!}{k!}e^{k\,x}\,\sum_{i=0}^{n}{{{k^{n-i}\,{{k}\choose{i}}\,x^{k-i}}\over{
 \left(n-i\right)!}}}.
$$

Hence from theorem \ref{CompositionOGF} and in view of the fact that these are self-inverse functions we obtain
$$
W^{(n)}(f(x))=\left\{
\begin{array}{ll}
\frac{1}{B_{1,1}}& n=1,\\
-\sum_{k=1}^{n-1} B_{n,k}W^{(k)}(f(x)), & n>1.
\end{array}
\right.
$$

Now we can write
$$
W^{(n)}=\left\{
\begin{array}{ll}
\frac{1}{1+x}e^{-x}& n=1,\\
-\frac{e^{-n\,x}}{\left(x+1\right)^{n}}{{n!\,\sum\limits_{m=1}^{n-1}e^{m\,x}\frac{W^{(m)}}{m!}{{{\sum\limits_{j=1}^{m}{\left(-1\right)^{j-m}\,{{m
 }\choose{j}}\,\sum\limits_{i=0}^{n}{{{j^{n-i}\,{{j}\choose{i}}\,x^{m-i}
 }\over{\left(n-i\right)!}}}}}}}}} & n>1
 \end{array}
\right. 
$$

Presented below are the first terms for the derivative

$$\frac{1}{1+x}e^{-x}$$  
$$\frac{-x-2}{(1+x)^3}e^{-2x}$$
$$\frac{(2\,x^2+8\,x+9)}{(1+x)^{5}e^{-3x}}$$
$$\frac{(-6\,x^3-36\,x^2-79\,x-64)}{(1+x)^7}e^{-4x}$$
$$\frac{24\,x^4+192\,x^3+622\,x^2+974\,x+625}{(1+x)^9}e^{-5x};$$
from whence we can obtain an expression for coefficients of the sequence A042977 \cite{OEIS}.

\section{Conclusion}

For derivation of the Bell polynomial of the second kind for the generating function $Y(x,z)=y(x+z)-y(x)$, it is necessary to use the composita of the generating function that can be obtained:\\
1) directly from the expression $Y(x,z)$ through transformations;\\
2) from theorem (\ref{FaDeeBruno}--\ref{theorem_inversion}).\\
Next, using formula (\ref{BellFormula}), the desired polynomial is derived. The numerous examples considered in the paper convincingly prove the efficiency of the proposed methods.

\end{document}